\documentclass[11pt]{amsart}

\usepackage{hyperref}
\usepackage[utf8]{inputenc}
\usepackage{latexsym,amssymb,amsmath}
\usepackage{mathrsfs}
\usepackage{epsfig}
\usepackage{amsfonts}
\usepackage{amssymb,amscd}
\usepackage{graphicx}
\usepackage{subfigure}
\usepackage{psfrag}
\usepackage{color}
\usepackage{comment}
\usepackage[all]{xy}\xyoption{rotate}\xyoption{pdf}
\usepackage{tikz-cd}

\RequirePackage{graphics}

\usepackage{bbm}
\usepackage[english]{babel}

\usepackage[left,modulo]{lineno}

\newtheorem{thm}{Theorem}
\newtheorem{prop}{Proposition}

\newtheorem{lem}[thm]{Lemma}

\theoremstyle{definition}
\newtheorem{defi}{Definition}[section]
\newtheorem{rem}{Remark}[section]

\newcommand\C{\mathbb{C}}

\newcommand\Z{\mathbb{Z}}
\newcommand\SL{\textrm{SL}}
\newcommand\PSL{\textrm{PSL}}

\newcommand\PGL{\textrm{PGL}}
\newcommand\G{\Gamma}
\newcommand{\Hom}{\mathrm{Hom}}

\def\square{\hfill${\vcenter{\vbox{\hrule height.4pt \hbox{\vrule width.4pt
height7pt \kern7pt \vrule width.4pt} \hrule height.4pt}}}$}

\author{E. Falbel, A. Guilloux}
\thanks{This work was supported in part by the ANR through the project "Structures G\'eom\'etriques et Triangulations".}
\address{Institut de Math\'ematiques de Jussieu-Paris Rive Gauche \\
Unit\'e Mixte de Recherche 7586 du CNRS \\
CNRS UMR 7586 and INRIA EPI-OURAGAN \\
Universit\'e Pierre et Marie Curie \\
4, place Jussieu 75252 Paris Cedex 05, France \\}
\email{elisha.falbel@imj-prg.fr\\ antonin.guilloux@imj-prg.fr }

\begin{document}

\title[] {Dimension of character varieties for $3$-manifolds}

\begin{abstract}
Let $M$ be a $3$-manifold, compact with boundary and $\Gamma$ its fundamental group.
Consider a complex reductive
 algebraic group $G$. The character variety $X(\Gamma,G)$ is
the GIT quotient $\Hom(\G,G)//G$ of the space  of morphisms $\G\to G$ by the natural action by conjugation of $G$. In the case $G=\SL(2,\C)$ this space has been thoroughly studied.

Following work of Thurston \cite{Thurston}, as presented by Culler-Shalen \cite{CullerShalen}, we give
a lower bound for the dimension of irreducible components of $X(\G,G)$ in terms of
the Euler characteristic $\chi(M)$ of $M$, the number $t$ of torus boundary components of $M$,
the dimension $d$ and the rank $r$ of $G$. Indeed, under mild assumptions on an
irreducible component $X_0$ of $X(\G,G)$, we prove the inequality 
$$\mathrm{dim}(X_0)\geq t \cdot r - d\chi(M).$$

\end{abstract}

\maketitle

\section*{Introduction}

Representation varieties of fundamental groups of 3-manifolds have been studied for a long time. 
 They are important
on one hand, to understand geometric structures on 3-manifolds and their topology as
 Thurston's work \cite{Thurston} has shown and,
on the other hand, to obtain more refined topological information on 3-manifolds as
 in the Culler-Shalen theory \cite{CullerShalen}.  
 
 The dimension of the character variety of $\pi_1(M)$ on $\PSL(2,\C)$ (where $M$ is a cusped manifold) at the representation corresponding to a complete finite volume hyperbolic structure is exactly the number of cusps (see  \cite{NeumanZagier} where it is shown that, in fact, that representation is a smooth point of the character variety). 
 A   general bound was established by Thurston and made more precise by Culler-Shalen: the dimension of irreducible components containing irreducible representations with
 non-trivial boundary holonomy is bounded below by the number of cusps.  In fact there are examples of
 irreducible representations with arbitrary large dimensions associated to hyperbolic manifolds with even only one cusp. %(see \cite{Long})
 
 The proof of the bound given in both Thurston and Culler-Shalen uses trace identities for $\SL(2,\C)$ which are difficult to work with in the higher dimensional case. % (see \cite{Lawton, Will} for the case $\SL(3,\C)$).
 
 In this paper we obtain a bound on the dimension of irreducible components which contain representations
 satisfying reasonable genericity conditions.  They include a notion of irreducibility and one of boundary regularity.  
 
 We first explain our result for $\SL(n,\C)$.  Although a particular case of our general result, it is
 simpler to state and contains the essential idea.  We prove the bound assuming that the image of the representation is Zariski dense and that the image of the fundamental group of the boundary  is regular in the sense that each image of a  boundary fundamental group has centralizer of minimal dimension (equal to the rank $n-1$ of $\SL(n,\C)$).  Note that the boundary regularity condition in the $\SL(2,\C)$ case simply means that the image is not central for each boundary component.
 
 We consider next representations into complex reductive affine algebraic groups. We use the general
 definitions of irreducibility as in Sikora \cite{Sikora} and the notion of regularity in reductive groups following Steinberg \cite{Steinberg} and prove the main Theorem \ref{thm:general}.

A special case of the theorem occurs when  $M$ is a cusped hyperbolic manifold and $\rho : \pi_1(M)\rightarrow \SL(n,\C)$ is the representation obtained by composing a representation $\pi_1(M)\rightarrow \SL(2,\C)$ obtained from the hyperbolic structure with the irreducible embedding  $\SL(2,\C)\rightarrow \SL(n,\C)$.   That representation is called a geometric representation and 
the dimension of the irreducible component of the character variety  $X(\G,\SL(n,\C))$ containing it
was first described in \cite{MP} (see also \cite{BFGKR, Guilloux}). 
Evidence that this bound was valid for all components containing irreducibles was obtained in computations of the character variety for the figure eight knot fundamental
group into $\SL(3,\C)$ \cite{FGKRT,HMP}.  In this case, irreducible components containing irreducible representations are all two dimensional.  Moreover each component contains Zariski-dense and  boundary regular unipotent representations already computed in \cite{falbeleight} (see also \cite{FKR14, CURVE} for other examples). 

One can use the results on local rigidity \cite{BFGKR} in order to prove that a component seen by the Deformation variety of an ideal triangulation (defined in \cite{BFGKR}) verifies this bound. 
We do not discuss this here, going straight to the proof which does not use triangulations.

As a final remark,  we should note that an a priori bound on the dimension of 
irreducible components certainly simplifies effective computations as in \cite{FGKRT} by eliminating the need to checking the existence of lower dimensional components.

We thank  Y. Benoist, P. Dingoyan, S. Diverio, P.-V. Koseleff, F. Rouillier, P. Will and M. Thistlethwaite for  many  discussions.

\section{A simple case}\label{s:SLn}

Before going to precise definitions in the framework of algebraic groups, let us sketch the proof
in a simple yet interesting case. Namely, in this section we assume $G=\SL(n,\C)$. It contains the original statement of Thurston (for the group $\SL(2,\C)$).
 Recall that its dimension $d$ is $n^2-1$ and its rank $r$ is
$n-1$: this is the minimal dimension of the centralizer of an element.

Let $M$ be a $3$-manifold,
compact with boundary. Let $t$ be the number of torus boundary components. We fix for each torus boundary
component an injection (denoted by an inclusion) $\pi_1(T)\subset \G$ . In other terms, for each of these torus boundary component, we choose a lift in the universal covering of $M$.

We want a lower bound on the dimension of components of $X(\G,G)=\Hom(\G,G)//G$. So consider an irreducible component $R_0$ of $\Hom(\G,G)$ and $X_0$ its projection in $X(\Gamma,G)$. We work with two assumptions on an element $\rho_0\in R_0$. One assumption should be a form of irreducibility for the whole representation
$\rho_0$. The second assumption is a regularity assumption for the image under $\rho_0$ of the fundamental groups of the torus boundaries. In the first version of our theorem, we assume:
\begin{itemize}
  \item Zariski-density: $\rho_0(\G)$ is Zariski-dense in $G$. 
  \item Boundary regularity: For any torus boundary component $T\subset \partial M$, 
  $\rho_0$ maps $\pi_1(T)$
  to a regular subgroup of $\SL(n,\C)$, i.e., 
   every subspace on which elements in $\rho_0(\pi_1(T))$ act by homothety is a line.
\end{itemize}

The second assumption is for example satisfied if the image of $\pi_1(T)$ is diagonalizable and every  
global eigenspace is a line, or if this image is unipotent and fixes a unique flag. It implies that the number of invariant subspaces is finite. It implies moreover that the centralizer of the image of $\pi_1(T)$ has dimension $n-1$.

We get the following particular case of the general theorem:
\begin{thm}
Let $X_0$ be the projection in $X(\G,G)$ of an irreducible component $R_0$ of $\Hom(\G,G)$ containing a representation $\rho_0$ which is Zariski-dense and boundary regular.

Then we have:
$$\dim(X_0)\geq (n-1)t-(n^2-1)\chi(M).$$
\end{thm}

Note that the case of $\rho_0$ being the geometric representation is not handled by this theorem, as 
the representation is not Zariski-dense. It will be handled by the 
general case, see theorem \ref{thm:general}.

\begin{proof}
Let us sketch the proof, inspired by Thurston and Culler-Shalen \cite{CullerShalen}. It works by induction on the 
number $t$ of boundary tori.

\noindent{\bf Initialization: $t=0$.} The inequality $\dim(X_0)\geq -(n^2-1)\chi(M)$ in this case is 
very general and was already known by Thurston. We give a proof in section \ref{s:notorus}.  Note that 
if $\chi(M)$ is non-negative the formula does not carry any information.

\noindent{\bf Propagation: $t-1\to t$.} Let $T$ be a torus boundary in $\partial M$. We first use
\begin{lem}\label{lem:drill}
There exists an element $\gamma$ of $\Gamma$ such that $\rho_0(\pi_1(T))$ and $\rho_0(\gamma)$ generate an irreducible representation of the subgroup of $\Gamma$ generated by $\pi_1(T)$ and $\gamma$.
\end{lem}

\begin{proof}
By the boundary regular assumption, $\rho_0(\pi_1(T))$ has a finite number of stable subspaces. By Zariski-density of $\rho_0(\Gamma)$, we find an element $\gamma\in\Gamma$ such that $\rho_0(\gamma)$ does not stabilize any of these subspaces.

Hence $\rho_0(\gamma)$ and $\rho_0(\pi_1(T))$ spans an irreducible subgroup of $G$.
\end{proof}

\begin{rem}
Using section 4.4 in \cite{Tits}, one may prove the above lemma (and then the theorem) with a mildest
irreducibility assumption and a different kind of regularity: namely that the Zariski-closure of $\rho_0$ is almost simple and irreducible, as e.g. the image of $SL_2(\C)$ under the irreducible representation and the boundary torus is mapped on a non trivial diagonal subgroup. This would be applicable to the geometric representation.

However we will not precisely discuss this, as this will be handled by the general theorem.
\end{rem}

Now, following Thurston and \cite{CullerShalen}, we drill along $\gamma$ in $M$. We get a new $3$-manifold $M'$ with a new genus $2$ boundary 
component denoted by $\Sigma_2$ and $t-1$ torus boundary components. Let $\delta$ be 
a meridian of the new handle of $\Sigma_2$ ($\delta$ is an element of $\pi_1(\Sigma_2)$ whose normal 
closure $N(\delta)$ in $\G'$ is the kernel of the map $\G'\to\G$).   Denote by $\Gamma'$ its fundamental 
group. We have an inclusion $\pi_1(\Sigma_2)\subset \G'$ and a surjective homomorphism $\G'\to \G$.  Summarizing:
\begin{figure}
\begin{center}
\def\svgwidth{1.2\textwidth}
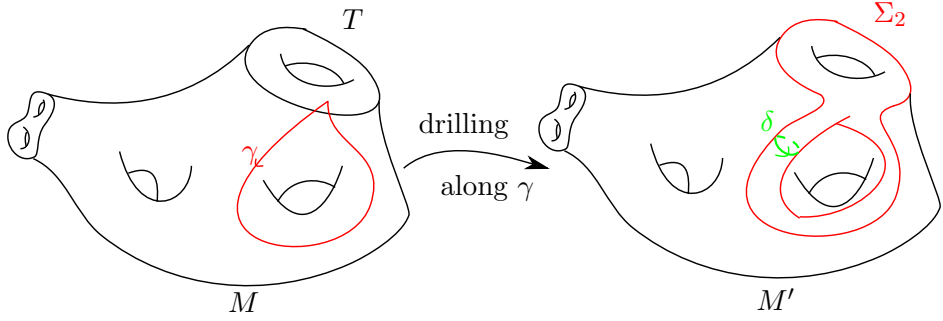
\end{center}
\caption{Drilling along $\gamma$}
\end{figure}
$$
\begin{CD}
1 @> >> N(\delta) @> >> \Gamma' @> >> \Gamma @> >>  1\\
@.@AA A @AA A@AA  A \\
1@> >> N(\delta)\cap \pi_1(\Sigma_2) @> >>\pi_1(\Sigma_2) @>  >> \pi_1(T) \star\Z @> >> 1\\
\end{CD}
$$
We obtain therefore a natural injection $\Hom(\G,\SL(n,\C))\hookrightarrow \Hom(\G',\SL(n,\C))$, that we see as an
inclusion. Let $V\subset \Hom(\G',G)$ be an irreducible component containing $R_0$. Then $R_0$ is an
irreducible component of the subvariety:
$$\left\{\rho \in V \mid \rho(\delta)=1\right\}.$$

Let $\mathrm{Res}$ be the restriction map:
$$\mathrm{Res}\: : \: \Hom(\G',\SL(n,\C)) \to \Hom(\pi_1(\Sigma_2),\SL(n,\C)).$$
We rephrase the previous remark on $R_0$ by saying that it is a component (in the variety $V$) of the preimage of the subvariety $$S=\left\{\tau \in \Hom(\pi_1(\Sigma_2),\SL(n,\C)) \mid \tau(\delta)=1\right\}$$
 of the variety $V'= \Hom(\pi_1(\Sigma_2),\SL(n,\C))$ by the regular map $\mathrm{Res}$.   We summarize the construction in the following maps:
 $$
\begin{tikzcd}
\Hom(\Gamma' , \SL(n,\C))\supset V \arrow[hookleftarrow]{r} \arrow{d}{\mathrm{Res}} &  \Hom( \Gamma,\SL(n,\C))\supset R_0 \arrow{d} \\
\Hom( \pi_1(\Sigma_2) ,\SL(n,\C)) = V' \arrow[hookleftarrow]{r} &  \Hom( \pi_1(T) \star \Z,\SL(n,\C))\supset S\\
\end{tikzcd}
$$

 We now use:
\begin{lem}\label{lem:surfaces}
The codimension of $S$ in $V'$ is at most $n^2-n$. Moreover $\mathrm{Res}(\rho_0)$ is a smooth point of both $S$ and $V'$.
\end{lem}

\begin{proof}
It is not difficult to prove that under our hypotheses, using fundamental results on the character varieties of surface groups (see Theorem \ref{thm:goldman}).
A general version of this lemma will be stated in section \ref{s:surfacegroups}.

Observe first that $\tau_0$ is irreducible, hence a smooth point.  The dimension at $\tau_0=\mathrm{Res}(\rho_0)$ of $V'$ is $3(n^2-1)$ and $\tau_0$. 

Secondly, the dimension of $S$ equals the dimension of $\Hom(\pi_1(T),\SL(n,\C))$, which is
$n^2+n-2$ (here we use the regularity hypothesis: the dimension of the centralizer of $\tau_0(\pi_1(T)$ 
is the rank $r=n-1$) plus the dimension of the representations of $<\gamma>$ which is $n^2-1$. As it is the minimal 
possible dimension (again by the regularity hypothesis), it is a smooth point.

Computing the difference, we get $n^2-n$.
\end{proof}

A crucial yet elementary point is the following proposition:

\begin{prop}\label{fact:codim}
Let $f : V\to V'$ a regular map between irreducibe algebraic varieties and $S\subset V'$ a subvariety. Let $p\in V$ be such that $f(p)$ belongs to $S$ and moreover is smooth for both $S$ and $V'$.
Let $\mathrm{codim}(S)$ be the maximal codimension of an irreducible component of $S$ in $V'$. 

Then, any
irreducible component of $f^{-1}(S)$ containing $p$ has codimension at most $\mathrm{codim}(S)$.
\end{prop}

\begin{proof}
Let $C$ be a component of $f^{-1}(S)$ containing $p$. Up to restricting to an irreducible component of $S$ containing $f(C)$, we may assume $S$ to be irreducible. 

We note that the set $C_\mathrm{good}$ of points $p$ in $C$ whose image is a smooth point of both $S$ and $V'$ is Zariski-open. Indeed, the sets $V'_\mathrm{smooth}$ and $S_\mathrm{smooth}$ are Zariski-open subsets respectively of $V'$ and $S$. In turn, $S\cap V'_\mathrm{smooth}$ is a Zariski-open subset of $S$. Hence, the set $S_\mathrm{good}=S_\mathrm{smooth}\cap V'_\mathrm{smooth}$ is a Zariski-open subset of 
$S$. Taking its preimage by $f$, we get the subset $C_\mathrm{good}=f^{-1}(S_\mathrm{good})$ which becomes Zariski-open in $C$. 

Moreover, $C_\mathrm{good}$ is not empty as it contains the point $p$. As the subset $C_\mathrm{smooth}$ of smooth points of $C$ is another non-empty Zariski-open subset of $C$, we conclude that they intersect non-trivially. So, up to changing the point $p$ for a point in this intersection, we may assume that $C$ is smooth at $p$.

The map $f$ induces a Zariski tangent map $T_pf$ from $T_pV$ to $T_{f(p)}V'$ such that $T_pC=T_pf^{-1}(T_{f(p)}(S))$. This linear map $T_pf$ induces in turn an injective map from $T_pV/T_pC$ to $T_{f(p)}V'/T_{f(p)}S$. We deduce that: 
$$\mathrm{codim}_{T_pV}(T_pC)\leq \mathrm{codim}_{T_{f(p)}V'}(T_{f(p)}S).$$

Now, from the smoothness considerations we made, we get the equalities:
$$\mathrm{dim}(T_pC)=\mathrm{dim}(C), \quad \mathrm{dim}(T_{f(p)}S)=\mathrm{dim}(S)\textrm{ and }\mathrm{dim}(T_{f(p)}V')=\mathrm{dim}(V').$$ We do not have any smoothness information about $p$ inside $V$, but we still have the general property that $\mathrm{dim}(T_pV)\geq \mathrm{dim}(V)$.

Eventually, we get the desired inequality: 
$$\mathrm{codim}_V(C)\leq \mathrm{codim}_{T_pV}(T_pC)\leq \mathrm{codim}_{T_{f(p)}V'}(T_{f(p)}S)=
\mathrm{codim}_{V'}S.$$
\end{proof}

\begin{rem} 
The smoothness hypothesis above is essential, otherwise one needs other assumptions (cf. \cite{Mumford} Theorem 2 pg. 48).
\end{rem}

As a corollary, using Proposition \ref{fact:codim}, we conclude that the codimension of $R_0$ in $V$ 
is at most $n^2-n$.

By induction, (we are working at the level of representation variety, so there is an additional 
$n^2-1$):
\begin{eqnarray*}\dim(V)&\geq& (n-1)(t-1)-(n^2-1)(\chi(M')-1)\\&\geq&(n-1)t-(n^2-1)(\chi(M)-1)+n^2-n.
\end{eqnarray*}
The upper bound for the codimension of $R_0$ grants that:
$$\dim(R_0)\geq (n-1)t-(n^2-1)(\chi(M)-1).$$
Now projecting to $X(\G,G)$, we loose $n^2-1$ dimension (as $\rho_0$ is irreducible):
$$\dim(X_0)\geq (n-1)t-(n^2-1)\chi(M).$$
\end{proof}

\section{Background on character varieties and reductive groups}\label{s:charactervarieties}

We now proceed with the general result, beginning with some definitions and results on algebraic groups and regularity, character varieties and irreducibility and character varieties of surface groups. This will lead to a general theorem whose proof is essentially the same as before.
Let $G$ be an affine reductive algebraic group of dimension $d$ and whose center has dimension $z$.
Suppose $\Gamma$ is a finitely generated group.

\subsection{Character and representation varieties}

Our reference for character varieties is the paper of Sikora \cite{Sikora}.

\subsubsection{Generalities}

The set of all homomorphisms $\Hom(\G, G)$  is an affine variety.   The action of $G$ on 
$\Hom(\G, G)$ is algebraic but the set theoretical quotient $\Hom(\G, G)/G$ is not necessarily an affine variety.
For that reason one works with the categorical quotient $\Hom(\G, G)//G$.  It is an affine variety such that the projection
$$
\Hom(\Gamma, G) \rightarrow \Hom(\Gamma, G)//G
$$
is constant on $G$-orbits and such that any morphism from $\Hom(\G, G)$ to an algebraic variety, which is constant on orbits, factors through
the projection.  It exists with our hypothesis that $G$ is reductive.

\begin{rem}\label{remark:dropdown-dimension}
Let $R_0$ be an irreducible component of $\Hom(\G,G)$ and $X_0$ its projection to $X(\G,G)$. At a generic
point $\rho$, one may compute the dimension of $R_0$ and $X(\G,G)$ as the dimension of the Zariski tangent space. Let $Z(\rho)$ be the centralizer in $G$ of $\rho(\G)$. It is the subgroup of $G$ fixing $\rho$ when acting by conjugation.
We have then that $\dim(X_0)=\dim(R_0)-(\dim(G)-\dim(Z(\rho))$.

In any case, we have $$\dim(X_0) \geq \dim(R_0)-d+z.$$
\end{rem}

\begin{rem}\label{remark:Z+Z}
It is known that  $\Hom(\Z\oplus \Z,\SL(n,\C))$ is an irreducible variety \cite[Theorem C]{Ri}.  Also, as 
 $\Hom(\G_1\star \G_2,G)=\Hom(\G_1,G)\times \Hom(\G_2,G)$,
 $\Hom((\Z\oplus \Z)\star \Z ,\SL(n,\C))$ is an irreducible variety.
\end{rem}

\subsubsection{Dimension of representation varieties for surface groups}\label{s:surfacegroups}

The most studied representation variety is the case where $\G$ is a surface group.  In particular, dimensions of character varieties 
were computed by Weil (\cite{Weil}).  We refer to Goldman's paper for a modern presentation. 

\begin{thm}[Thm 1.2 in \cite{Goldman}]\label{thm:goldman}
Let $G$ be a reductive group, $\pi$ be the fundamental group of a closed oriented surface of genus $g$ 
and $\tau$ belong to $\Hom(\pi,G)$. Let $Z(\tau)$ denote the centralizer of the image of $\tau$.
Then the 
dimension of the Zariski tangent space at $\tau$ to the space $\Hom(\pi,G)$ is
$$(2g-1)\mathrm{dim}(G)+\mathrm{dim}(Z(G)).$$ 
Moreover, if the representation $\tau$ is irreducible then $\tau$ is a smooth point.

\end{thm}

\subsection{Irreducibility and regularity in algebraic groups}\label{s:IrreducibilityRegularity}

In this section, we review hypotheses needed to replace the so-called "irreducibility" and "regularity" of the case $G=\SL(n,\C)$. We want to deal with semisimple or reductive groups. The natural framework is then the rather general theory of algebraic groups.
%Our reference is the paper of Sikora \cite{Sikora}

\subsubsection{Reductive groups}

Let $G$ be the points over $\mathbb C$ of a reductive algebraic group. 
Let $d$ be its dimension, $z$ the dimension of its center and $r+z$ the minimal dimension of the centralizer of an element in $G$: $r$ is the rank of the semisimple part $G/Z(G)$, or the dimension of a maximal torus included in this semisimple part.

Recall that in the cases of $\PGL(n,k)$ or $\SL(n,k)$, we have $d=n^2-1$, $z=0$, $r=n-1$.

\subsubsection{Irreducibility and strong irreducibility}

We recall the general notion of irreducibility, as given in \cite{Sikora}:
\begin{defi}[Irreducibility]\label{def:irreducible}
We say that $\rho\in \Hom(\G,G)$ is \emph{irreducible} if its image is not
contained in any parabolic subgroup of $G$.
\end{defi}

Recall from \cite{Sikora} the following property:
\begin{prop}[Prop. 13 in \cite{Sikora}]
If $\rho$ is irreductible, the centralizer in $G$ of $\rho(\G)$ is the center of $G$.
\end{prop}

We will need a slightly stronger notion of irreducibility known as strong irreducibility:
\begin{defi}[Strong irreducibility]
We say that $\rho\in \Hom(\G,G)$ is \emph{strongly irreducible} if the image of none of its finite index subgroups is contained in any parabolic subgroup of $G$.
\end{defi}

Let us note that this is a quite natural and usual notion in the context of discrete subgroups of $\SL(n,\C)$ or linear representations where it is rephrased as the fact that no strict subspace has a finite orbit under the group. Several appearances of this notion can be given, from divisible convex sets in \cite{Benoist} to random walks in groups \cite{GR} and finite groups representations \cite{Gross} or problems of commensurability of subgroups in the special linear group \cite{long-reid}.

\subsubsection{Regularity}

\begin{defi}[Regularity]\label{def:regular}
Let $H$ be a commutative group in $G$. We say that $H$ is regular if its centralizer in $G$
has the minimal dimension $r+z$.
\end{defi}

Now let $M$ be a compact  $3$-manifold with boundary, $\G$ its fundamental group. As before, for each torus $T$ we fix an inclusion $\pi_1(T)\subset \G$. We will always assume that the boundary of $M$ is not empty.

\begin{defi}[Boundary Regularity]\label{def:boundary-regular}
A representation $\rho\in \Hom(\G,G)$ is  \emph{boundary regular} if the image $\rho(\pi_1(T))$ is regular for each torus component $T\in \partial M$.
\end{defi}

For $G=\PGL(n,k)$ or $\SL(n,k)$, (the subgroup generated by) a single element is regular if and only if
all its eigenspaces are lines. A semisimple element (in the framework of algebraic group) is regular but a regular unipotent (as in \cite{Guilloux}) is regular too.

\section{The general case}

We may now state the general theorem:
\begin{thm}\label{thm:general}
Let $G$ be a complex reductive group of dimension $d$ and rank $r$ and whose centre has dimension $z$.

Let $M$ be a compact $3$-manifold, whose boundary is non-empty and contains $t$ torus components. 
Denote by $\G$ its fundamental group. Let $R_0$ be an irreducible component of $\Hom(\G,G)$ 
containing a strongly irreducible and boundary regular representation $\rho$. Let $X_0$ be its projection in the character variety $X(\G,G)$.  Then
$$\dim(X_0) \geq r.t - d\chi(M)+z.$$
\end{thm}

\subsection{A general version of key lemmas}

\begin{lem}\label{lem:drillgeneral}
Let $\rho$ be a strongly irreducible, boundary regular representation of $\G$ in $G$ and $T$ be a torus boundary component.

Then there exists $\gamma\in\Gamma$ such that the subgroup of $G$ generated by $\rho(\pi_1(T))$ 
and $\rho(\gamma)$ is irreducible.
\end{lem}

\begin{proof}
Modding out by the center of $G$, we may assume that $G$ is semisimple (the 
irreducibility assumption is preserved, see \cite[Lemma 11]{Sikora}).
We obtain the lemma by using results on regular elements given by Steinberg \cite{Steinberg}. 

By hypothesis, the Zariski-closure $G'$ of $\rho(\Gamma)$ is not contained in any parabolic
subgroup of $G$ and neither is any of its finite index subgroup. Moreover the subgroup $H=\rho(\pi_1(T))$ is a regular subgroup of $G$. From Steinberg \cite{Steinberg}, we know that it is included in a finite number of Borel subgroups of $G$.  On the other hand, for each Borel subgroup, there is only a finite number of maximal parabolic subgroups
of $G$ containing it \cite[Section 0.8]{Humphreys}. Hence $H$ is included in a finite number of maximal parabolic subgroups of G.

Consider the set $A$ of elements $g\in G$ such that the 
subgroup generated by $g$ and $H$ is not contained in any parabolic subgroup of $G$. Its complement is 
the union of the finite number of maximal parabolic subgroups of $G$ containing $H$.
subgroups of $G$. The intersection of each of these parabolic subgroups with $G'$ is a closed subgroup 
of $G'$. Moreover, as $G'$ is neither contained nor virtually contained in any parabolic subgroup, this closed subgroup is a proper parabolic subgroup of $G'$ of positive codimension (it is not a finite index subgroup). 
Hence the complement of $A$ is Zariski-closed in $G$ and $A\cap G'$ is the intersection of a finite number of non empty Zariski-open subsets of $G'$ whose complement has positive codimension.

As $G'$ is the Zariski-closure of $\rho(\Gamma)$, there exists an element $\gamma\in \Gamma$ such that 
$\rho(\gamma)$ belongs to $A$.  Therefore $\rho(\gamma)$ and $\rho(\pi_1(T))$ are not included in any parabolic subgroup: the group they generate is irreducible.
\end{proof}

\begin{lem}\label{lem:surfacesgeneral}
Let $\tau_0$ be an irreducible representation of $\pi_1(\Sigma_2)$ in $G$ and with a symplectic basis 
of homology $\alpha,\beta,\gamma,\delta$. Suppose $\tau_0(\delta)=1$ and $\tau_0(<\alpha,\beta>)$ is 
regular.

Then the codimension in $\Hom(\pi_1(\Sigma_2),G)$ of the subvariety $S$ defined by $\tau(\delta)=1$ is 
at most $d-r$. Moreover, $\tau_0$ is a smooth point in both the subvariety $S$ and $\Hom(\pi_1(\Sigma_2),G)$.
\end{lem}

\begin{proof}
By \cite[Prop 1.2]{Goldman} reviewed before (see Theorem \ref{thm:goldman}), the irreducible representation $\tau_0$ is a smooth point of $\Hom(\pi_1(\Sigma_2),G)$ and the dimension at this point is $3d+z$.

Let us now compute the dimension of the subvariety $S$ defined by $\tau(\delta)=1$. A point in this variety is 
given by the image of the group generated by $\alpha$, $\beta$ and $\gamma$. Moreover, from the presentation of 
$\pi_1(\Sigma_2)$ and the condition $\tau(\delta)=1$, we get that the images of $\alpha$ and $\beta$ 
commute. So the dimension of $S$ is the dimension at $(\tau_0)_{|\langle \alpha,\beta\rangle}$ of
the space of representations of $\mathbb Z^2$ -- i.e. the fundamental group of the torus -- plus the 
dimension of the representations of $\gamma$. This gives (applying again theorem \ref{thm:goldman}) 
that the dimension of the Zariski-tangent subspace at $\tau_0$ to the subvariety $S$ is $(d+r+z)+d$. As it is the minimal possible dimension for a Zariski-tangent subspace of $S$ (by the regularity hypothesis), we get that $\tau_0$ is a smooth point of $S$.

We conclude that the codimension is at most $d-r$ as announced.
\end{proof}

\subsection{When no boundary component is a torus}\label{s:notorus}

We can follow step-by-step the presentation of Culler-Shalen~\cite{CullerShalen} in this case. We get:
\begin{prop}[No torus boundary component]
Let $M$ be a compact $3$-manifold whose boundary is non empty ans does not contain a torus component.
 Let $G$ be a real or complex algebraic group of dimension $d$. Then for any irreducible component $R_0$ of
$Hom(\Gamma,G)$ and $X_0$ its projection in $X(\G,G)$, we have :
$$\mathrm{dim}(X_0)\geq - d \chi(M)+z.$$
\end{prop}

\begin{proof}
Following remark \ref{remark:dropdown-dimension}, we want to prove $\mathrm{dim}(R_0)\geq -d\chi(M)+d$.

We assumed that $\partial M\neq \emptyset$. Then $M$ has the homotopy type of a 
finite $2$-dimensional CW-complex with one $0$-cell, $m_1$ $1$-cells and $m_2$
$ 2$-cells. The Euler characteristic of $M$ is then $1-m_1+m_2$.

This decomposition gives a presentation
$$\G=\langle g_1,\,\ldots,\, g_{m_1} \: \mid \: r_1, \, \ldots, \, r_{m_2} \rangle.$$
In other terms, $\Hom(\G,G)$ is the preimage under a regular map from $G^{m_1}$ to $\G^{m_2}$
of the point $(1,\ldots,1)\in G^{m_2}$. As an application of fact \ref{fact:codim} the dimension of any irreducible component -- e.g. $R_0$ -- is greater than
$n\cdot(m_1-m_2)=-d\chi(M)+d$.
\end{proof}

\subsection{Conclusion}\label{ss:general}

Replacing both lemmas of the proof in the simple case by their general version above, replacing $n-1$ by $r$ and $n^2-1$ by $d$, we readily get the announced lower bound:
$$\dim(X_0) \geq r.t  - d\chi(M)+z.$$

\bibliographystyle{amsalpha}
\bibliography{bibli}

\newcommand{\etalchar}[1]{$^{#1}$}
\providecommand{\bysame}{\leavevmode\hbox to3em{\hrulefill}\thinspace}
\providecommand{\MR}{\relax\ifhmode\unskip\space\fi MR }
\providecommand{\MRhref}[2]{%
  \href{http://www.ams.org/mathscinet-getitem?mr=#1}{#2}
}
\providecommand{\href}[2]{#2}
\begin{thebibliography}{FGK{\etalchar{+}}14}

\bibitem[Ben05]{Benoist}
Yves Benoist, \emph{Convexes divisibles. {III}}, Ann. Sci. \'Ecole Norm. Sup.
  (4) \textbf{38} (2005), no.~5, 793--832.

\bibitem[BFG{\etalchar{+}}13]{BFGKR}
N.~Bergeron, E.~Falbel, A.~Guilloux, P.~V. Koseleff, and F.~Rouillier,
  \emph{Local rigidity for {$\mathrm{SL}(3,\bC)$} representations of
  $3$-manifold groups}, Exp. Math. \textbf{22} (2013), no.~4, 410--420.

\bibitem[CS83]{CullerShalen}
M.~Culler and P.~B. Shalen, \emph{Varieties of group representations and
  splittings of $3$-manifolds.}, Ann. of Math. (2) \textbf{117} (1983), no.~1,
  109--146.

\bibitem[CUR]{CURVE}
\emph{{\large CURVE}}, http://curve.unhyperbolic.org/.

\bibitem[Fal08]{falbeleight}
E.~Falbel, \emph{A spherical {C}{R} structure on the complement of the figure
  eight knot with discrete holonomy}, Journal of Differential Geometry
  \textbf{79} (2008), 69--110.

\bibitem[FGK{\etalchar{+}}14]{FGKRT}
E.~Falbel, A.~Guilloux, P.-V. Koseleff, F.~Rouillier, and M.~Thistlethwaite,
  \emph{Character varieties for sl(3,c): the figure eight knot.},
  arXiv:1412.4711 (2014).

\bibitem[FKR14]{FKR14}
E.~Falbel, P-V. Koseleff, and F.~Rouillier, \emph{Representations of
  fundamental groups of 3-manifolds into {$\mathrm{PGL}(3,\bC)$}: Exact
  computations in low complexity}, Geometriae Dedicata (2014).

\bibitem[Gol84]{Goldman}
William~M. Goldman, \emph{The symplectic nature of fundamental groups of
  surfaces}, Adv. in Math. \textbf{54} (1984), no.~2, 200--225. \MR{762512
  (86i:32042)}

\bibitem[GR89]{GR}
Yves Guivarc'h and Albert Raugi, \emph{Propri\'et\'es de contraction d'un
  semi-groupe de matrices inversibles. {C}oefficients de {L}iapunoff d'un
  produit de matrices al\'eatoires ind\'ependantes}, Israel J. Math.
  \textbf{65} (1989), no.~2, 165--196.

\bibitem[Gro90]{Gross}
Benedict~H. Gross, \emph{Group representations and lattices}, J. Amer. Math.
  Soc. \textbf{3} (1990), no.~4, 929--960.

\bibitem[Gui13]{Guilloux}
A.~Guilloux, \emph{Deformation of hyperbolic manifolds in {$\PGL(n,\mathbb C)$}
  and discreteness of the peripheral representations}, Proc. of the AMS (2013),
  To appear.

\bibitem[HMP15]{HMP}
M.~Heusener, V.~Munoz, and J.~Porti, \emph{The sl(3,c)-character variety of the
  figure eight knot}, arXiv:1505.04451 (2015).

\bibitem[Hum95]{Humphreys}
James~E. Humphreys, \emph{Conjugacy classes in semisimple algebraic groups},
  Mathematical Surveys and Monographs, vol.~43, American Mathematical Society,
  Providence, RI, 1995. \MR{1343976 (97i:20057)}

\bibitem[LR99]{long-reid}
DD~Long and AW~Reid, \emph{Commensurability and the character variety},
  Mathematical Research Letters \textbf{6} (1999), no.~5/6, 581--592.

\bibitem[MFP12]{MP}
P.~Menal-Ferrer and J.~Porti, \emph{Local coordinates for
  {$\mathrm{SL}(n,\bC)$}-character varieties of finite volume hyperbolic
  $3$-manifolds}, Ann. Math. Blaise Pascal \textbf{19} (2012), no.~1, 107--122.

\bibitem[Mum76]{Mumford}
D.~Mumford, \emph{Algebraic geometry, i}, Springer-Verlag, Berlin-New York,
  1976, Complex projective varieties, Grundlehren der Mathematischen
  Wissenschaften, No. 221.

\bibitem[NZ85]{NeumanZagier}
Walter~D. Neumann and Don Zagier, \emph{Volumes of hyperbolic three-manifolds},
  Topology \textbf{24} (1985), no.~3, 307--332. \MR{815482 (87j:57008)}

\bibitem[Ric79]{Ri}
R.~W. Richardson, \emph{Commuting varieties of semisimple {L}ie algebras and
  algebraic groups}, Compositio Math. \textbf{38} (1979), no.~3, 311--327.

\bibitem[Sik12]{Sikora}
A.~S. Sikora, \emph{Character varieties}, Trans. Amer. Math. Soc. \textbf{364}
  (2012), no.~10, 5173--5208.

\bibitem[Ste65]{Steinberg}
Robert Steinberg, \emph{Regular elements of semisimple algebraic groups}, Inst.
  Hautes \'Etudes Sci. Publ. Math. (1965), no.~25, 49--80. \MR{0180554 (31
  \#4788)}

\bibitem[Thu80]{Thurston}
W.~Thurston, \emph{The geometry and topology of 3-manifolds},
  http://library.msri.org/books/gt3m/.

\bibitem[Tit72]{Tits}
J.~Tits, \emph{Free subgroups in linear groups}, J. Algebra \textbf{20} (1972),
  250--270. \MR{0286898 (44 \#4105)}

\bibitem[Wei60]{Weil}
A.~Weil, \emph{On discrete subgroups of lie groups}, Ann. of Math. (2)
  \textbf{72} (1960), 369--384.

\end{thebibliography}
\end{document}